\documentclass[12pt]{article}
 \usepackage{amsmath}

\usepackage[square]{natbib}
 \usepackage{amssymb,amsthm}

\newcommand{\blind}{0}

\newtheorem{thm}{Theorem}[section]
\newtheorem{lem}{Lemma}[section]
\newtheorem{pro}{Proposition}[section]
\newtheorem{cor}{Corollary}[section]

\newtheorem{rem}{Remark}[section] 

\theoremstyle{definition}

\newtheorem{defn}{Definition}[section]

\DeclareMathOperator{\tr}{Tr}
\DeclareMathOperator{\E}{\mathbb{E}}

\DeclareMathOperator{\vc}{Vec}
\DeclareMathOperator{\cov}{Cov}
\DeclareMathOperator{\R}{\mathbb{R}}
\DeclareMathOperator{\N}{\mathbb{N}}
\DeclareMathOperator{\LL}{\mathcal {L}}
\DeclareMathOperator{\ra}{\rightarrow}

\begin{document}

\def\spacingset#1{\renewcommand{\baselinestretch}%
{#1}\small\normalsize} \spacingset{1}


\if0\blind
{
  \title{\bf Fisher Information and Exponential Families \\
 Parametrized by a Segment of Means }
  \author{
   Piotr Graczyk \hspace{.2cm}\\
     LAREMA, Universit\'e d'Angers\\
     Salha Mamane  \\
   School of Statistics and Actuarial Science,  University of the Witwatersrand \\
   and  
     LAREMA, Universit\'e d'Angers
    }
\date{}

  \maketitle

\begin{abstract}
 We consider natural and general exponential families $(Q_m)_{m\in M}$ on $\mathbb{R}^d$ parametrized by the means.
 We study the submodels $(Q_{\theta m_1+(1-\theta)m_2})_{\theta\in[0,1]}$  parametrized by a segment in the means domain, mainly from the point of view of the Fisher information.
Such a parametrization allows for a parsimonious 
model and is particularly useful in practical situations when hesitating between two  parameters $m_1$ and $m_2$.
 The most interesting examples are obtained when $\mathbb{R}^d$ is a linear space of matrices, in particular for  Gaussian and Wishart models.
\end{abstract}

\footnote{
Supported by l'Agence Nationale de la Recherche
ANR-09-BLAN-0084-01.

AMS Subject Classification: 62H12, 62H10  

{\it Keywords:} Fisher information, efficient estimator, exponential family,   multivariate Gaussian distribution,
 Wishart distribution, 
 parsimony.}

\section{Introduction}

Fisher information is a key concept in  mathematical statistics. Its importance  stems from the Cram\'er-Rao inequality
which says that the variance of any unbiased estimator $T(X_1,\hdots X_n)$ of an unknown parameter $\theta$, is bounded by
the inverse of the Fisher information:
$\mathrm{Var}_{\theta}(T)-(I(\theta))^{-1}$ is 
semi-positive definite. Fisher information is therefore mainly used as a measure of how well a parameter can be a 
estimated.  This justifies the use of Fisher information in experimental design 
for predicting the maximum precision an experiment can provide on model parameters.
This also justifies the important role  Fisher information plays  in estimation theory where it provides  bounds for
 confidence regions and also in Bayesian analysis where it  provides a basis for noninformative priors.
 Fisher information can  be used to 
investigate the trade-off between parsimony of parameters and precision of the estimation of the parameters 
\citep{andersson2006ieee}.

Besides its importance in statistical theory, Fisher information has different interpretations that lead to some 
practical applications.  For example, the interpretation
of  Fisher information as a measure of the state of disorder of a dynamic system leads  to the use of Fisher information in stochastic optimal
control as a tuning tool to stabilise the performance of a dynamic system \citep{Ramirez}.
 Viewing  Fisher information as a measure of
information, leads to the statement of a ``minimum information principle'' akin to the well-known maximum entropy
principle for determining the ``maximally unpresumptive distribution'' satisfying some predefined constraints
\citep{Bercher}. \cite{gupta} describe the use of  Fisher information in model selection as a tool to discriminate between two models with
otherwise very similar fit to some data.
The use of Fisher information however goes far beyond statistics; \cite{frieden} shows that Fisher information
is in fact a key concept in the unification of science in general, as it allows a systematic approach to deriving 
Lagrangians.\\

The objective of this work is the study of the Fisher information for  exponential models $(P_m)_{m\in M}$ parametrized by a segment of  means $[m_1,m_2]$.
Exponential families of distributions have been extensively studied  \citep{brown,barndorff,letac_lectures, letac_casalis}.
 A parametrization of the family by a segment  instead of the whole means domain allows 
to obtain a  parsimonious model when the mean domain is high-dimensional. The parametrization of the mean parameter by a segment  is
particularly useful in practical situations when hesitating between two equally convenient mean values
 $m_1$ and $m_2$. 
 Such parametrization will also serve in sequential data collection, when an updated estimate of a
parameter largely differs from the previous estimate.
An important practical example is a Gaussian  model $N(u, \theta C +D)$ in  $\mathbb{R}^d$   with the mean vector $u$ known
 and the covariance matrix  in a segment $ IC+D$, where $\theta\in I=[a,b]\subset \R$.

 From the Fisher information point of view, exponential families constitute an interesting and important class of models. 
Their Fisher information  coincides with 
the second derivative of the cumulant generating function of the measure generating the family
and they are the only models for which the Cram\'er-Rao bound can
always be attained \citep{brown,letac_lectures}. \\

The paper is organised as follows. 
In Section \ref{general},  basic definitions and results 
on Fisher information and exponential families are recalled and 
extended to matrix-parametrized models.
 Section \ref{segment} contains
   new results on the  Fisher information of  exponential families parametrized by the domain of the means and the sub-families 
   parametrized by a segment of means $[m_1,m_2]$. 
In Section \ref{appli}, these  results are
 applied to   Gaussian and Wishart families of distributions.  When $m_1$ and $m_2$ are colinear, we construct efficient estimators for the segment parameter $\theta$.

\section{Preliminaries}\label{general}
In most expositions of the theory of  exponential families and of the concept of Fisher information,
 the  parameter  is considered to be a vector whereas cases abound in multivariate analysis where 
the canonical parameter is   a matrix.  In this preliminary section we adapt the presentation
of the usual objects of exponential families (mean function, variance function and Fisher information) 
to the case where the canonical parameter is a matrix. 

We denote by $\mathbb{R}^{k\times m}$  the space of real matrices with $k$ rows and $m$ columns and by  $A\otimes B$
  the Kronecker product of two matrices.
 We use the usual notations  $A^T$ for the transpose matrix and  $\langle  A,B \rangle=\tr(A^TB)$ for the scalar product of two matrices.
 The operator $\vc$   converts a $k\times m$ matrix $A$ into a vector
 $\vc(A)\in \mathbb{R}^{k m}$
 by 
stacking the columns one underneath the other. The  $\vc$ operator is commonly used in applications of the matrix differential calculus in statistics, cf.
\citep{matdiff, Muir}. 

The following properties of the Kronecker product  are  used  in this work \citep[p.32,35]{matdiff}.
For non-singular squared matrices $A$, $B$ we have 
$(A\otimes B)^{-1}= A^{-1}\otimes B^{-1}$. 
For all matrices $A$, $B$ and $C$ such that the product $ABC$ is well defined
\begin{equation}\label{kron}
         \vc(A\,B\,C)= (C^T\otimes A) \vc(B).
\end{equation}
In this paper we use the following convention of the matrix differential calculus:
if a function $f:\mathbb{R}^{k\times p} \rightarrow  \mathbb{R}^{n\times m}$ is differentiable
then its derivative is a matrix $f'(x)\in \mathbb{R}^{nm\times kp}$ such that
 \begin{equation}\label{magnus}
\vc(df(x)(u)) =f'(x)\vc(u),\ \ \ u\in \mathbb{R}^{k\times p}.
 \end{equation}
The only  exception we will  make 
is the  derivative of a  function $K:\mathbb{R}^{k\times m}\rightarrow\mathbb{R}$,
 for which the following convention is  used:
 the derivative of $K$ is not a row vector but the matrix $K'(x)\in\mathbb{R}^{k\times m}$, 
related to the differential of $K$ by $dK(x)(u)=\langle K'(x), u\rangle = \tr(K'(x)^T u)$,
for all $u\in \mathbb{R}^{k\times m}$. 
This convention is needed to give sense to  formula (\ref{meanX}) for the mean of an exponential family.

In this section we consider  probability models  $\left(P_{s}(d\omega) \right)_{s \in S}$ ,
 $S \subset \mathbb{R}^{k\times m}$,  
on a measurable space $(\Omega, \mathcal{A})$ 
such that there exists a $\sigma$-finite positive  measure $\nu$ on $(\Omega, \mathcal{A})$ and a real
 function $(\omega, s)\mapsto l_{\omega}(s)$ such that 
\begin{equation*}
 P_{s}(d \omega)=e^{l_{\omega}(s)}\nu(d\omega)
\end{equation*}
and $\int e^{l_{\omega}(s)}\nu(d\omega)=1$.
These models  encompass but are not reduced to exponential families of distributions. We suppose that $S$ is open
and that the function $s\rightarrow l_{\omega}(s)$ is twice differentiable. We impose on $l_{\omega}(s)$
classical regularity conditions, allowing double   differentiation  under the integral sign in  $\int e^{l_{\omega}(s)}\nu(d\omega)$.

The score function $l^{\prime}_{\omega}(s)$ is the   derivative with respect 
to $s$ of the log-likelihood function $l_{\omega}(s)$. It is a $k\times m$ matrix with  zero mean.
The extension of the definition of Fisher information matrix from vector-parametrized models to 
 matrix-parametrized models is straightforward.
\begin{defn}\label{fisher}
 The Fisher information matrix of the model $\left(P_{s}(d\omega) \right)_{s \in S}$, 
$S\subset \mathbb{R}^{k\times m}$
  on a measurable space $(\Omega, \mathcal{A})$ is the $km\times km$ symmetric matrix 
\begin{eqnarray*}
 I(s)= \cov(l^{\prime}_{\omega}(s))
=\int_{\Omega} \vc(l^{\prime}_{\omega}(s)) \vc(l^{\prime}_{\omega}(s))^TP_{s}(d\omega).
\end{eqnarray*}
\end{defn}
Similarly as for vector parametrized models, the Fisher information can be written as the negative of the mean 
of the second derivative of the log-likelihood function.
\begin{pro}\label{lbis}
The Fisher information  of the model $\left(P_{s}(d\omega) \right)_{s \in S}$, 
$S\subset \mathbb{R}^{k\times m}$
equals
\begin{equation*}
 I(s)=-\int_{\Omega} l''_{\omega}(s)P_{s}(d\omega).
\end{equation*}
\end{pro}
\begin{proof}
 As a probability distribution, $ P_{s}(d\omega)$ satifies
\begin{eqnarray*}
 \int_{\Omega} P_{s}(d\omega)=1 &\implies& \frac{d}{ds}\int_{\Omega} P_{s}(d\omega)
=\int_{\Omega}\frac{d}{ds} P_{s}(d\omega)
= \int_{\Omega} e^{l_{\omega}(s)}l^{\prime}_{\omega}(s)\nu(d\omega)=0_{\mathbb{R}^{k\times m}}\\
&\implies& \frac{d}{ds}\int_{\Omega}e^{l_{\omega}(s)} l^{\prime}_{\omega}(s)\nu(d\omega)
= \int_{\Omega} \frac{d}{ds}\left(e^{l_{\omega}(s)} l^{\prime}_{\omega}(s)\right)\nu(d\omega)
=0_{\mathbb{R}^{k m\times k m}}\\
&\implies& \int_{\Omega} \left(e^{l_{\omega}(s)} \vc(l'_{\omega}(s))\vc(l^{\prime}_{\omega}(s))^T
+ e^{l_{\omega}(s)} l''_{\omega}(s) \right)\nu(d\omega)=0_{\mathbb{R}^{k m\times k m}}
\end{eqnarray*}
and the Proposition follows.
\end{proof}
Consider a $\sigma$-finite measure space $(\Omega, \mathcal{A},\nu)$, a finite-dimensional normed vector space $E$ and a measurable map
 $T:\Omega \rightarrow E$.
Let $\mu$  be the image  of the measure $\nu$ by $T$ and 
let $S_{\mu}$ be the interior of the  domain of the moment generating function of $\mu$, i.e. the set
$\{ s\in E:  \int_{E} e^{\langle s , x \rangle}d\mu(x) < \infty \}$.
 We suppose  that $\mu$ is a  $\sigma$-finite measure on $E$ and that $S_{\mu}\not=\emptyset$. 
The cumulant generating function of $\mu$  is defined as the logarithm of the moment generating function of $\mu$:
\begin{equation*}
k_{\mu}(s)=\log \int_{E} e^{\langle s , x \rangle}\mu(dx)=\log \int_{\Omega} e^{\langle s , T(\omega) \rangle}\nu(d\omega),\ \  s\in S_{\mu}.
\end{equation*}

\begin{defn}
The general exponential family generated by the measure $\nu$ and the map $T$ is the family of probability distributions
\begin{equation}\label{param0}
\{ P(s,T,\nu)(d\omega)= e^{\langle s \,,\, T(\omega)\rangle-k_{\mu}(s)}\nu(d\omega): \quad s\in S_{\mu}\}.
\end{equation}
The natural exponential family associated with the above general exponential family is the family
of probability distributions defined on the space $E$ by
\begin{equation}\label{param1}
 P(s,\mu)(dx)= e^{\langle s \,,\, x\rangle-k_{\mu}(s)}\mu(d x), \quad s\in S_{\mu}.
\end{equation}
\end{defn}
 Natural exponential families may be viewed as a special case of  general exponential families   with $\Omega=E$, $T(\omega)=\omega$ and $\nu=\mu$.
The following result is well-known for  vector-valued and matrix-valued exponential families \citep{letac_casalis}. 
\begin{pro}
 \begin{enumerate}
  \item The set $S_{\mu}$ is convex. If $\mu$ is not concentrated on some affine hyperplane of $E$, then $k_{\mu}$ is a strictly convex function on $S_{\mu}$. 
\item The map $s \mapsto k_{\mu}^{\prime}(s)$ is an analytic diffeomorphism from $S_{\mu}$ to
its image $M=k_{\mu}^{\prime}( S_{\mu})\subset \mathbb{R}^{k\times m} $ called the domain of the  means of the family. In particular $M$ is open.
 \end{enumerate}
\end{pro}
The name "domain of the  means" for the set $M$ is justified by  formula (\ref{meanX}) of the following  Proposition, well-known in the vector case.
\begin{pro}
The mean and covariance of a random matrix $X$ following the distribution $P(s,\mu)$ belonging to the natural exponential family  generated by 
a measure $\mu$ are  given by
\begin{eqnarray}
 m(s) &=& \mathbb{E}_s(X) = k_{\mu}'(s) \label{meanX}\\
v(s)&=& \cov_s(\vc(X))= k_{\mu}''(s). \label{covX} 
\end{eqnarray}
\end{pro}
\begin{proof}
Formula  (\ref{meanX}) follows from
\begin{eqnarray*}
 k_{\mu}'(s)=\frac{\int_{S_{\mu}}xe^{\langle s , x \rangle}\mu(dx)}{\int_{S_{\mu}}e^{\langle s , x \rangle}\mu(dx)}
=\int_{S_{\mu}}xe^{\langle s , x \rangle-k_{\mu}(s)}\mu(dx)
=\mathbb{E}_s(X).
\end{eqnarray*}
Next, using (\ref{meanX}), we obtain  
\begin{eqnarray*}
  k_{\mu}''(s)&=& \int_{S_{\mu}}\frac{d}{ds}\left( x e^{\langle s , x \rangle-k_{\mu}(s)}\right)\mu(dx)
= \int_{S_{\mu}} \vc(x) \vc( x-k'_{\mu}(s))^Te^{\langle s , x \rangle-k_{\mu}(s)}\mu(dx)\\
&=& \int_{S_{\mu}} \vc(x-m(s)) \vc( x-m(s))^Te^{\langle s , x \rangle-k_{\mu}(s)}\mu(dx)
= \cov_s(\vc(X)).
\end{eqnarray*}
\end{proof}
\begin{rem}\label{momentsT}
If $W$ is a random matrix with a law $P(s,T,\nu)$ from the general exponential family,
then $T(W)=X$ in law and $m(s)$ and $v(s)$ are the mean and the covariance of $T(W)$.
\end{rem}

Now we compute the Fisher  information of the exponential families parametrized by the canonical parameter $s\in S_\mu$.
\begin{pro}\label{FisherP}
The Fisher  information for the parameter $s$ of  exponential families  (\ref{param0}) and  (\ref{param1})
is given by
\begin{equation}\label{Is}
 I(s)= k''_{\mu}(s)=v(s).
\end{equation}
\end{pro}
\begin{proof}
The log-likelihood is equal to
$l_{\omega}(s)={\langle s \,,\, T(\omega)\rangle}-k_{\mu}(s)$, so
 $l''_{\omega}(s)=-k''_{\mu}(s)$ does not depend on $\omega$.
Formula (\ref{Is}) follows by  Proposition \ref{lbis}.
\end{proof}

\begin{defn}
 Denote by $\psi: M \rightarrow S_{\mu}$, $m \mapsto \psi(m)= {k_{\mu}^{\prime}}^{-1}(m)$  the inverse of the
 diffeomorphism $k_\mu^{\prime}$. 
The general exponential family, parametrized by the domain of
the means $M$ is  given by the family of distributions
\begin{equation}\label{param3}
 Q(m,T,\nu)(d\omega)= e^{\langle \psi(m) \,,\, T(\omega)\rangle-k_{\mu}(\psi(m))}\nu(d\omega), \quad m\in M.
\end{equation}

The natural exponential family, parametrized by the domain of the means $M$ is the family
of probability distributions defined on the space $E$ by

\begin{equation}\label{param2}
  Q(m,\mu)(dx)= e^{\langle \psi(m) \,,\, x\rangle-k_{\mu}(\psi(m))}\mu(dx), \quad m\in M .
\end{equation}
\end{defn}
The mean of the families (\ref{param3}) and  (\ref{param2})  is equal to $m$. We denote the covariance
 of the families (\ref{param3}) and  (\ref{param2})  by $V(m)$ and we have by (\ref{covX})
 \begin{equation}\label{var_fn}
V(m)=v(\psi(m))=k''(\psi(m)).
 \end{equation}
The function $V: m \in M \rightarrow V(m)$ is called the {\it variance function} of the exponential family.\\

We will compute the Fisher information of the exponential families  (\ref{param3}) and  (\ref{param2}) parametrized by the mean
$m\in M$ in the next section.  We will need the following  formula  giving the Fisher information for a reparametrized model. 
\begin{thm}\label{c1}
 Consider a model $\left(P_{s}(d\omega) \right)_{s\in S}$ and  a reparametrization
$f: {\tilde S\subset \mathbb{R}^{k\times p} \rightarrow S \subset \mathbb{R}^{n\times m}}$, where $f$ is a differentiable map.
Let $I(s)$ be the information matrix of $\left(P_{s}(d\omega) \right)_{ s\in S}$.
The Fisher information matrix of the  model 
$\left(Q_t(d\omega) \right)_{t \in \tilde S}=\left(P_{f(t)}(d\omega) \right)_{t \in \tilde S}$ is 
\begin{equation}\label{reparametr}
\tilde I(t)=f'(t)^T I(f(t))f'(t).
\end{equation}
\end{thm}
\begin{proof}
Let us denote $h_{\omega}(t)=l_{\omega}(f(t))$. We have  $Q_t(d\omega)=e^{h_{\omega}(t)}\nu(d\omega).$
 For all $t\in \tilde S$ and $u\in \mathbb{R}^{k\times p}$,
\begin{eqnarray*}
 dh_{\omega}(t)(u) &=& dl_{\omega}(f(t))(df(t)(u))=
\langle l'_{\omega}(f(t)) , df(t)(u) \rangle = \\
&=& \vc\left(l'_{\omega}(f(t)) \right)^T\vc\left( df(t)(u)\right) = \vc\left(l'_{\omega}(f(t)) \right)^T f'(t)\vc\left( u \right).
\end{eqnarray*}
Thus, using the convention introduced after (\ref{magnus}), 
$\vc\left(h'_{\omega}(t)\right)^T = \vc\left(l'_{\omega}(f(t)) \right)^T f'(t)$  and 
$$\vc\left(h'_{\omega}(t)\right) \vc\left(h'_{\omega}(t)\right)^T
=f'(t)^T\vc\left(l'_{\omega}(f(t)) \right)\vc\left(l'_{\omega}(f(t)) \right)^T f'(t).$$
Therefore, by Definition \ref{fisher} we get $\tilde I(t)=f'(t)^T I(f(t))f'(t)$. 
\end{proof}
\section{Fisher information of exponential families parametrized by the mean}\label{segment}
In this section we first compute the Fisher information of the exponential families  (\ref{param3}) and  (\ref{param2}) parametrized by the mean.
Next we consider the same problem for a submodel parametrized by a segment of means.
In order to avoid confusion, when the parameter of an exponential family is the mean $m$  we will denote the Fisher information
by $J(m)$. 
\begin{thm}\label{th1}
 The Fisher information of  the exponential families  (\ref{param3}) and  (\ref{param2}) parametrized by the mean
$m\in M$ equals 
\begin{equation}\label{info_m}
 J(m)= V(m)^{-1}=\psi'(m),
\end{equation}
where  $V(m)$ is  the variance function of the exponential family, given by (\ref{var_fn}).
\end{thm}
\begin{proof}
 We use Theorem \ref{c1} with $f=\psi: M\rightarrow S_\mu$. 
Since $\psi(m)={k'_{\mu}}^{-1}(m)$, we have $\psi'(m)=[k''_{\mu}(\psi(m))]^{-1}$.
Thus $J(m)=[k''_{\mu}(\psi(m))]^{-1}k''_{\mu}(\psi(m))[k''_{\mu}(\psi(m))]^{-1}=[k''_{\mu}(\psi(m))]^{-1}=V(m)^{-1}.$
\end{proof}
\begin{rem}
 Note a striking  contrast in the formulas (\ref{Is}) and (\ref{info_m}) for the Fisher information of an exponential family parametrized
 either by the canonical parameter $s\in S_\mu$ or by the mean $m\in M$; in the first case we have $I(\psi(m))=V(m)$, in the second
  $J(m)=V(m)^{-1}$.
\end{rem}

\subsection{Fisher information of exponential families parametrized by a \\segment of means}

 Consider a general exponential family $\{Q(m,T,\nu)(d\omega)\ : \ m\in M\}$ parametrized by the domain of the means $M$.
 Let $A\not=0,B\in \mathbb{R}^{k\times m} $ be two matrices. Define  $\Theta=\{\theta\in \mathbb{R}: \theta A+B\in M\}$. The set  $\Theta\subset \mathbb{R}$ is open because
 $M$ is open. 
 Suppose that $\Theta\not=\emptyset$. The  parametrization by a segment  of means $I\subset\Theta$      consists in considering the submodel 
 \begin{equation}\label{segm}
\{Q(\theta A +B,T,\nu): \theta\in  I\}.
 \end{equation}


In statistical practice, the following situation will be concerned by such models. Let $m_1\in M$ and $m_2\in M$ be
    two different estimations of the true mean $m$ of an exponential family (\ref{param3}) or (\ref{param2}) .
When one hesitates between them as estimators, and when $M$ is convex, it is natural to consider the model
$$
\{Q(\theta m_1 +(1-\theta)m_2 ,T,\nu) :  \theta \in [0,1]\}.
$$
Writing $\theta m_1 +(1-\theta)m_2=\theta(m_1-m_2) +m_2$ we see that this is a special case of the model (\ref{segm}).

The following theorem  gives the Fisher information of a general exponential
family parametrized by a segment of means. By analogy to the notation $J(m)$, we denote this information by $J(\theta)$. 

\begin{thm}\label{thx}
 The Fisher information of the model $\{Q(\theta A+B,T,\nu): \theta\in I\}$ equals
\begin{equation}\label{Jtheta}
J(\theta)=\vc(A)^T\, V(\theta A+B)^{-1}\,\vc(A).
\end{equation}
\end{thm}
\begin{rem}
 This and the following results are also true for submodels  $\{Q(\theta A+B,\mu): \theta\in I\}$ of  natural exponential families.                                                                                                                                                                                       
\end{rem}

\begin{proof}
By Theorem \ref{th1}, the Fisher information of the  model  $\{Q(m,T,\nu)\ : m\in M\}$ is
 $J(m)= V^{-1}(m)$.
  We apply Theorem \ref{c1} to  the reparametrization $f:I\ra M,  f(\theta)=\theta A +B$. 
We  have $f'(\theta)=\vc(A).$  
Then formula (\ref{reparametr})  gives (\ref{Jtheta}).
\end{proof}

The following Lemma is useful for the  derivation of an alternative formula for the Fisher information of an exponential family parametrized by a
segment of means and verifying an additional condition (\ref{eq20}). We will see in   Section \ref{appli} that this condition holds for Gaussian and Wishart models.
\begin{lem}\label{const}
 Assume that for all $m\in M$,
\begin{equation}\label{eq20}
         \langle m\,,\,\psi(m)\rangle= C,
        \end{equation}
for some constant $C\in \mathbb{R}$. Then, for all $u\in M$,
\begin{equation}\label{lemma}
  \langle m \,,\, d\psi(m)(u)\rangle  = -\langle u\,,\, \psi(m)\rangle.
\end{equation}
\end{lem}
\begin{proof} By (\ref{eq20})   the differential of the function $g:M  \rightarrow \mathbb{R}$,
$m \mapsto \langle m\,,\, \psi(m)\rangle$ is  zero.
 Therefore, for all $m,u\in M$
\begin{eqnarray*}
  dg(m)(u)=  \langle m \,,\, d\psi(m)(u)\rangle + \langle u\,,\, \psi(m)\rangle =0\\
\end{eqnarray*}
and (\ref{lemma}) follows.
\end{proof}
\begin{cor}\label{corx}
 Let  $\{Q(\theta A+B,T,\nu)(d\omega): \theta\in I\}$ be an exponential
 model parametrized by a segment of means.
 If  the condition (\ref{eq20}) holds then 
  the Fisher information of the  model   equals
 \begin{equation}
 J(\theta)=-\frac{d^2}{d\theta^2} \left[k_{\mu}(\psi(\theta A+B))\right].
 \end{equation}
\end{cor}

\begin{proof}
Let $h(\theta)= k_{\mu}(\psi(\theta A +B))$ and $f(\theta)=\theta A +B$. 
We want to compute $h''(\theta)$.
 If $\theta, u \in \mathbb{R}$, 
\begin{eqnarray*}
dh(\theta)(u) &=& dk_{\mu}(\psi(f(\theta)))\big(d\psi(f(\theta))(df(\theta)(u))\big)\\
&=& \langle k'_{\mu}(\psi(f(\theta)))\,,\,d\psi(f(\theta))(df(\theta)(u))  \rangle\\
&=& \langle f(\theta)\,,\,d\psi(f(\theta))(df(\theta)(u)) \rangle\\
&=& -\langle df(\theta)(u) \,,\,\psi(f(\theta))  \rangle\\ 
&=& -u\langle  A\,,\,\psi(f(\theta))  \rangle, 
\end{eqnarray*}
where we used successively: the convention on $k'_{\mu}$ introduced after (\ref{magnus}), the equality $k'_{\mu}\circ \psi(m)=m$,
Lemma \ref{const} and the formula $ df(\theta)(u)=uA$. Thus we have  
$h'(\theta)=
 -\langle  A\,,\,\psi(f(\theta))  \rangle.$ 
 Now, starting as in the computation of $h'(\theta)$ and using (\ref{magnus}), we get
 $$h''(\theta)=   -\langle  A\,,\,d\psi(f(\theta))(A)  \rangle=-\vc(A)^T\vc(d\psi(f(\theta))(A))=
  -\vc(A)^T \psi'(\theta A + B) \vc(A).$$
We conclude using (\ref{info_m}) and
  Theorem \ref{thx}. 
\end{proof}

\section{Applications}\label{appli}

In this section, we apply the  results from  preceding sections 
to the study of  some important exponential families parametrized by a segment of means.

We denote by $\mathcal{S}_d$ the vector  space of $d \times d$ symmetric matrices and by  $\mathcal{S}_d^+$ 
the open cone of positive definite matrices.
\subsection{Exponential families of Gaussian distributions}
Let us recall the construction of the multivariate Gaussian model $\{N(u, \Sigma);\, \Sigma \in \mathcal{S}_d^+ \}$ as a general exponential family.
 We consider $\Omega=\mathbb{R}^d$ equiped with a normalised Lebesgue measure  
  $\nu(d\omega)=d\omega/(2\pi)^{d/2}$, the space 
 $E=\mathcal{S}_d $  and the map
  $$ T: \mathbb{R}^d  \rightarrow \mathcal{S}_d,\ \ T(\omega)=-\frac12 {(\omega-u)(\omega-u)^T}.$$
  The image measure
 $\mu$ on $E$  is concentrated on   the opposite of the cone of  semi-positive definite matrices of 
rank one.  
 For  $s\in \mathcal{S}_d$, the moment generating function of $\mu$ equals  
$$
\int_{\Omega}e^{\langle s,T(\omega)\rangle}\nu(d\omega)=\frac{1}{(2\pi)^{d/2}}\int_{\mathbb{R}^d}
e^{ -\frac12\tr\left(s {(\omega-u)(\omega-u)^T} \right)}d\omega
=(\det s)^{-1/2}
$$
when  $s\in\mathcal{S}_d^+ $  and it is infinite  otherwise. Thus   $S_{\mu}=\mathcal{S }_d^+ $
and the cumulant function is
  $$k_{\mu}(s)=-\frac{1}{2}\log \det(s),\ \ \ s\in S_{\mu}=\mathcal{S}_d^+. $$
  The general exponential family is therefore 
\begin{eqnarray} \label{eqt1}
P(s,T,\nu)(d\omega)=
\frac{1}{(2\pi)^{d/2}}e^{\langle s \,, \,-\frac{1}{2}(\omega-u) (\omega-u)^T\rangle+\frac{1}{2}\log \det(s)}d\omega
=\frac{(\det s)^{1/2}}{(2\pi)^{d/2} } e^{-\frac{1}{2}(\omega-u)^T s (\omega-u)}d\omega,
\end{eqnarray}
which is the family of Gaussian distributions $N(u, s^{-1})$ on $\R^d$ with a  fixed mean $u\in \R^d$, parametrized by $s=\Sigma^{-1}$,  the inverse of the covariance matrix $\Sigma$ supposed to be invertible.

The derivative of the function $X\in \R^{d\times d}\ra \det X$ is the cofactor matrix $X^\sharp$ which equals $(\det X)(X^{-1})^T$ when
$X$ is inversible. It follows that 
 $$m(s)=k_{\mu}'(s)=-\frac{1}{2}s^{-1}, \ \ \  s\in \mathcal{S}_d^+.$$
 This can be also  deduced from Remark \ref{momentsT};   if $W$ is a random vector with law $N(u, s^{-1})$,
then   $$m(s)=k_{\mu}'(s)= \E T(W)=\E(-\frac12 {( W-u)(W-u)^T}) =-\frac12 \cov W\in -\mathcal{S}_d^+.$$
 The means domain   is $M=-\mathcal{S}_d^+ $  and the inverse mean map is
$\psi(m)=-\frac{1}{2}m^{-1}$.
 The  Gaussian general exponential family parametrized by $m$ in the  means domain $M=-\mathcal{S}_d^+$ is therefore the  family  
 \begin{equation}\label{GaussMean}
  Q(m, T,\nu)=N(u, -2m).
 \end{equation}
 Up to a trivial affine change of parameter $\Sigma=-2m$, this parametrization by the covariance parameter  is more natural  than
 the  parametrization of the family  $(N(u, s^{-1}))_{s\in \mathcal{S}_d^+}$   by the canonical parameter $s$.
 
In order to compute the variance function, recall that $XX^{-1}=I_d$ implies that $dX^{-1}=-X^{-1}dX\,X^{-1}$
and $(X^{-1})'=-X^{-1}\otimes \,X^{-1}$. Thus $k''_\mu(s)=\frac{1}{2}s^{-1} \otimes s^{-1}$ and  formula (\ref{var_fn}) implies that
\begin{equation}\label{varianceGauss}
  V(m) =2m\otimes m.
\end{equation}

By Proposition \ref{FisherP}, 
  the Fisher information of the family  $(N(u, s^{-1}))_{s\in \mathcal{S}_d^+}$  is
 $I(s)= \frac{1}{2} s^{-1}\otimes s^{-1}.$
 
 By Theorem \ref{th1} and  formula (\ref{varianceGauss}), 
  the Fisher information of the model $(N(u, -2m))_{m\in - \mathcal{S}_d^+}$ equals
  $J(m)= \frac{1}{2}m^{-1}\otimes m^{-1}$.
\begin{cor}\label{sigma}
 The Fisher information matrix of the Gaussian model $(N(u, \Sigma))_{\Sigma \in \mathcal{S}_d^+}$
is  
$$J(\Sigma)= \frac{1}{2}\Sigma^{-1}\otimes \Sigma^{-1}.$$
\end{cor}
\begin{proof}
 Using Theorem \ref{c1} and a reparametrization $\Sigma=-2m$  we obtain $\tilde J(\Sigma)= \frac{1}{2}\Sigma^{-1}\otimes \Sigma^{-1}= J(\Sigma)$.
\end{proof}

Let us now consider 
Gaussian models parametrized  by a segment of covariances.
\begin{cor}
Let $C$ and $D$ be two symmetric matrices and let $I\subset \R$ be a non-empty segment 
such that 
 $I\subset  \Theta = \{\theta\in \mathbb{R}: \theta C+D\in \mathcal{S}_d^+\}$.
 The Fisher information   of the Gaussian model $\{N(u, \theta  C  + D),  \theta \in I\}$
is 
\begin{eqnarray*}
 J(\theta)&=& \frac{1}{2}\tr\left(C(\theta C +D)^{-1} C(\theta C +D)^{-1}\right).
\end{eqnarray*}
\end{cor}
\begin{proof}
We use Corollary \ref{sigma} and Theorem \ref{c1} with $f(\theta)=\theta C +D$. It follows that
$$
 J(\theta)=  \vc(C)^T J(\theta C +D) \vc(C)=
 \frac{1}{2} \vc(C)^T \left((\theta C +D)^{-1} \otimes (\theta C +D)^{-1}  \right)\vc(C).
 $$
 Applying (\ref{kron}) we get 
 $$
  J(\theta)= \frac{1}{2} \vc(C)^T \vc\left((\theta C+D)^{-1} C (\theta C+D)^{-1} \right) 
= \frac{1}{2}\tr\left(C(\theta C+D)^{-1} C(\theta C +D)^{-1}\right).
 $$
\end{proof}
On the other hand  we have 
the following  alternative formula for the information $J(\theta)$. 
\begin{cor}
 The Fisher information   of the Gaussian model $\{N(u, \theta  C  + D),  \theta \in I\}$
is 
\begin{eqnarray}\label{infoGaussBis}
 J(\theta)= -\frac12 \frac{d^2}{d\theta^2} (\log\det(\theta C+D)).
\end{eqnarray}
\end{cor}
\begin{proof}
Observe that the condition (\ref{eq20}) holds for the Gaussian exponential families  $Q(m,t,\nu)$:
$$ \langle m\,,\,\psi(m)\rangle=-\frac12\tr(mm^{-1})=-\frac{d}{2}.$$ 
The model 
$N(u, \theta  C  + D)=N(u,-2m)=Q(m,T,\nu)$ with $m=\theta A +B\in M=-\mathcal{S}_d^+$ where $A=-\frac{C}{2} $ and $B=-\frac{D}{2} $.
 We apply  Corollary \ref{corx} and the fact that
 $$k_{\mu}(\psi(\theta A+B))=-\frac12\log\det(\theta C+D).$$
 Formula (\ref{infoGaussBis}) follows.
\end{proof}

Now we characterize  the information $J(\theta)$ in terms of the eigenvalues of
the matrix $D^{-1/2}CD^{-1/2}$.
\begin{thm}\label{eigenvals}
Let $C$ and $D$ be two symmetric  matrices and let $I\subset \R$ be a segment 
such that  $I C+D\subset \mathcal{S}_d^+$.
Let $a_1,\hdots,a_d$ be the eigenvalues of the matrix $D^{-1/2}CD^{-1/2} $.

 The Fisher information   of the Gaussian model $\{N(u, \theta  C  + D),  \theta \in I\}$
  equals
 \begin{equation}\label{infoEigenv}
J(\theta)
=\frac{1}{2}\sum_{j=1}^d\left(\frac{a_j}{1+a_j\theta}\right)^2.
\end{equation}
\end{thm}
\begin{proof}
The idea of the proof is to use  formula (\ref{infoGaussBis}).
Let $P(\lambda)$ be the characteristic polynomial of the matrix $D^{-1/2}CD^{-1/2}$. We have
$$P(\lambda)=\det (D^{-1/2}CD^{-1/2}-\lambda I_n)=\det(D^{-1}C-\lambda I_n)=(\det D)^{-1}\det (C-\lambda D) 
.$$
On the other hand 
$ P(\lambda)=\prod_{j=1}^n(a_j-\lambda)$. It follows that
$$\det(\theta C+D)=\det D\times \theta^{d}P(-1/\theta)=
 \det D\prod_{j=1}^{d}(\theta a_j+1).$$
 The last formula allows to compute easily $ \frac{d^2}{d\theta^2} (\log\det(\theta C+D))$. First we see that
 $$
  \frac{d}{d\theta} (\log\det(\theta C+D))=\frac{\frac{d}{d\theta}\det(\theta C+D)}{\det(\theta C+D)}=\sum_{j=1}^{d}\frac{a_j}{\theta  a_j+1}.
 $$
 One more derivation and  formula (\ref{infoGaussBis})  lead to  (\ref{infoEigenv}).
\end{proof}
We finish by computing the Fisher information of   two  Gaussian models in $\R^d$,  parametrized by an explicitely given  segment of covariances.
First, let $A$ be a circulant matrix with the first row $e_2+e_d=(0,1,0,\ldots,0,1)$. Then for a segment $I\subset \R$ containing 0 and  $\theta\in I$ 
 \begin{equation}\label{firstcase}
\theta  A+I_d=\begin{pmatrix}
1 & \theta & 0 &\hdots& 0 & \theta\cr
\theta & 1 & \theta & 0 &\hdots& 0\cr 
 0 & \theta & 1 & \theta & 0&\hdots\cr 
& & \ddots&\ddots & \ddots& \cr
0&\hdots & 0 & \theta & 1 & \theta\cr
 \theta & 0 &\hdots& 0 & \theta & 1
\end{pmatrix}\in  \mathcal{S}_d^+. 
\end{equation}
\begin{cor}\label{corcircul}
 The Fisher information of the model $\left(N(0 ,\theta A+I_d)\right)_{\theta\in I}$ is given by
\begin{eqnarray}\label{circulant}
 J(\theta)&=&\frac12\sum_{j=0}^{d-1} \left(\frac{2\cos(\frac{2\pi j}{d}) }{1+2\theta \cos(\frac{2\pi j}{d})}\right)^2.
\end{eqnarray}
\end{cor}
\begin{proof}
Let ${\mathcal A}$ be a circulant matrix with the first row $(r_0,r_1,\ldots, r_{d-1})$.
It is well known (see e.g.
\citep{circulant}) and easy to check that if $\epsilon$ is a $d$-th root of unity,
 $\epsilon^{d}=1$, then  
$a= \sum_{l=0}^{d-1} r_{l}\epsilon^{l}$ is an eigenvalue of  ${\mathcal A}$ with an eigenvector
$(1,\epsilon, \epsilon^2,\ldots,  \epsilon^{d-1})$.

Therefore if   $\epsilon_j=e^{\frac{2\pi j i}{d}}$, $j=0,\hdots,d-1$ are the $d$ distinct $d$-th roots of unity, then 
the matrix  ${\mathcal A}$ has $d$ distinct eigenvalues  $a_j= \sum_{l=0}^{d-1} r_{l}\epsilon_j^{l}$.
In our particular case,
\begin{eqnarray*}
     a_j= e^{\frac{2\pi j i }{d}}+e^{\frac{2(d-1)\pi j i }{d}}= 2\cos\left(\frac{2\pi j}{d}\right).
    \end{eqnarray*}
Formula (\ref{circulant}) follows from Theorem  \ref{eigenvals}. 
\end{proof}

Now, let us consider  a tridiagonal matrix $C$ such that
\begin{equation}\label{secondcase} 
 \theta C+I_d=\begin{pmatrix}
1 & \theta & 0 & 0 & 0&\hdots\cr 
\theta & 1 & \theta & 0 & 0&\hdots\cr 
0 & \theta & 1 & \theta & 0&\hdots\cr 
&\ddots&\ddots&\ddots&\ddots& \cr
0&\hdots & 0 & \theta & 1 & \theta\cr
0&\hdots & 0 & 0 & \theta & 1
\end{pmatrix}.
\end{equation}
\
As in the preceding case, there exists a segment $I\subset \R$ such that  $ \theta C+I_d\in  \mathcal{S}_d^+$ for $\theta\in I$.

\begin{cor}\label{cortridiag}
 The Fisher information of the model $\left(N(0 ,\theta C+I_d)\right)_{\theta\in I}$ is given by
\begin{equation}\label{tridiag}
 J(\theta)=
\frac12\sum_{j=1}^d\left(\frac{2\cos\left(\frac{j}{d+1}\pi \right)}{1+2\theta\cos\left(\frac{j}{d+1}\pi \right)}\right)^2.
\end{equation}

\end{cor}
\begin{proof}
We will apply Theorem \ref{eigenvals} with $C$ and  $D=I_d$. 
Expanding $\psi_d(\lambda)=\det\left(C-\lambda I_d\right)$ along the first row, we get
$\psi_d(\lambda)=-\lambda \psi_{d-1}(\lambda)-M^{1,2}$.
 Expanding the minor $M^{1,2}$ along its first column gives  $M^{1,2}=\psi_{d-2}(\lambda)$
and
\begin{equation*}
 \psi_d(\lambda)=-\lambda \psi_{d-1}(\lambda)-\psi_{d-2}(\lambda),\,\,d\geq 3.
\end{equation*}
 We set $\varphi_d(\lambda)=(-1)^d\psi_d(2\lambda)$ and  we obtain
\begin{equation*}
 \varphi_d(\lambda)=2\lambda \varphi_{d-1}(\lambda)-\varphi_{d-2}(\lambda),\,\,d\geq 3
\end{equation*}
with initial conditions
$\varphi_1(\lambda)=2\lambda$, $\varphi_2(\lambda)=4\lambda^2-1$.
 Therefore $\varphi_d$ is a Tchebyshev polynomial of the second kind  \citep{mason_chebyshev_2003} and it
satisfies $\varphi_d(\cos x)=\frac{\sin(d+1)x}{\sin x},\,\,d\geq 1$.

We have, for all $\lambda\in [-2 , 2]$, 
\begin{eqnarray*}
 \psi_d(\lambda)=0 \iff \varphi_d\left(\frac{\lambda}{2}\right)=0 &\implies& \frac{\sin(d+1)x}{\sin x}=0,\quad x=\arccos \frac{\lambda}{2}.\\
\end{eqnarray*}
Therefore $\lambda_j=2\cos\left(\frac{j}{d+1}\pi \right)$, $1\leq j\leq d$, are $d$ distinct eigenvalues of the 
matrix $C$.
\end{proof}

\subsection{Exponential families of Wishart distributions}

Let $E=\mathcal{S}_d$ be the space of symmetric real  matrices of order $d$. The Riesz measures $\mu_p$ on the cone
$\overline{\mathcal{S}_d^+}$ are unbounded positive measures such that their Laplace transform equals for $t\in \mathcal{S}_d^+$
$$
\LL\mu_p(t)= \int_{\overline{\mathcal{S}_d^+}} e^{-\langle t,x\rangle} d \mu_p(x)= (\det t)^{-p}.
$$
By the celebrated Gindikin theorem, such measures exist if and only if  $p$ belongs to the Gindikin set
 $\Lambda_d=\{\frac{1}{2},\hdots,\frac{d-1}{2}\} \cup \left(\frac{d-1}{2}, \infty\right)$. They are supported by 
 the  cone $\overline{\mathcal{S}_d^+}$  if and only if $p>\frac{d-1}{2}$ and   they are absolutely continuous in that case,
 with a density $\Gamma_d(p)^{-1}(\det x)^{p-\frac{d+1}{2}}$, $x\in \mathcal{S}_d^+$, $\Gamma_d(p)=\Gamma(p)\Gamma(p-\frac12)\ldots \Gamma(p-\frac{d-1}2)$.
 Otherwise, when $p\in \{\frac{1}{2},\hdots,\frac{d-1}{2}\}$,  the measures  $\mu_p$ are singular and concentrated on semipositive symmetric matrices of rank $2p$.
 
The family of   Wishart distributions $W(p;s)$ on $\overline{\mathcal{S}_d^+}$  is defined as the  natural  exponential family  generated by the Riesz measure
$\mu_p$. According to (\ref{param1}), it means that $p\in \Lambda_d$,   $s\in S_{\mu_p}= -\mathcal{S}_d^+$  and   
$$
W(p;s)(dx)= \frac{ e^{\langle s \,,\, x\rangle}}{\LL\mu_p(-s)}\mu_p(d x)=  e^{\langle s \,,\, x\rangle} (\det(-s))^p\mu_p(d x)=   e^{\langle s \,,\, x\rangle-k_{\mu_p}(s)}\mu_p(d x), 
$$
with $k_{\mu_p}(s)= -p \log \det(-s)$. It follows that $\LL W(p;s) (t)=\det(I_d+ (-s)^{-1}t)^{-p}$ and that $\mu_p(dx)=e^{\tr x} W(p;-I_d)$. 
 
 Wishart distributions are multivariate analogs of the gamma distributions 
 $\lambda^p\Gamma(p)^{-1}e^{-\lambda x}x^{p-1}dx$
 on $\R^+$( $p>0,\lambda>0)$, considered with a canonical parameter $s=-\lambda<0$.
 Similarly as in dimension 1,  the Wishart distributions are often parametrized by a scale parameter $\sigma= (-s)^{-1}\in \mathcal{S}_d^+$
and then the notation $\gamma(p;\sigma)=W(p; (-\sigma)^{-1})$ is used, cf. \citep{letac_noncentral}. 
 The study  of Wishart distributions is motivated by their importance as estimators of the covariance matrix of a Gaussian model in $\R^d$.\\
 
 Let us apply our results on the Fisher information to a   natural  exponential family  of   Wishart distributions $\{W(p;s):\  s\in -\mathcal{S}_d^+\}. $ 
 The mean equals $m(s)=k'_{\mu_p}(s)= p(-s)^{-1}\in M= \mathcal{S}_d^+$ and  the inverse mean map $\psi: \mathcal{S}_d^+\ra - \mathcal{S}_d^+$  is $\psi(m)=  -pm^{-1}$.
 
 Thus the Wishart family $Q(m,\mu_p)$ parametrized by the domain of means
 is, up to a trivial reparametrization $m\ra \frac1p m$, the family parametrized by its scale parameter:
 \begin{equation}\label{WishartMean}
 Q(m,\mu_p)=W(p; -pm^{-1})=\gamma(p;\frac1p m),\ \ \ m\in \mathcal{S}_d^+. 
  \end{equation}
 As $v(s)=k_{\mu_p}''(s) = p (s^{-1}\otimes s^{-1})$, it follows that
  the variance function 
 is
 \begin{equation}\label{VarWish}
  V(m)= \frac{1}{p} (m\otimes m).
 \end{equation}

 
 By Proposition \ref{FisherP},  the Fisher information of the model $\{W(p;s):\  s\in -\mathcal{S}_d^+\} $ is
  $I(s)= ps^{-1}\otimes s^{-1}$.
 By Theorem \ref{th1} the Fisher information of the model $\{Q(m,\mu_p),\; {m\in M}\}$
 is $J(m)= p m^{-1}\otimes m^{-1}$.
 
Consequently, using Theorem \ref{c1} and a reparametrization $m\ra \frac1p m=\sigma $  we  see
that the 
 the Fisher information matrix of the Wishart model $\{ \gamma(p;\sigma):\ \sigma\in \mathcal{S}_d^+ \}$
 parametrized by a scale parameter $\sigma$ equals
 $J(\sigma)=p \sigma^{-1}\otimes \sigma^{-1}$.

 \begin{thm}
 Let  $I=(a,b)\subset\R$ and $C,D\in \mathcal{S}_d$ such that $IC+D\subset\mathcal{S}_d^+$.
 The Fisher information $J(\theta)$ of the
 Wishart model $\{ \gamma(p; \theta C +D):\ \theta \in I\}$  verifies the formulas
 \begin{eqnarray}
  J(\theta)&=& p\tr\left(C(\theta C +D)^{-1}  \right)^2 \label{wish1}\\
   J(\theta)&=& - p \frac{d^2}{d\theta^2} (\log\det(\theta C+D))\nonumber\\
   J(\theta)
&=& p \sum_{j=1}^d\left(\frac{a_j}{1+a_j\theta}\right)^2 \label{wish2}
 \end{eqnarray}
 where $a_1,\hdots,a_d$ are the eigenvalues of the matrix $D^{-1/2}CD^{-1/2} $.
 \end{thm}
 \begin{proof}
  The proofs are similar to the proofs of the analogous results for exponential Gaussian families in the previous subsection.
  The condition  (\ref{eq20}) holds  true:
$ \langle m\,,\,\psi(m)\rangle= -pd$, the model  $\{ \gamma(p; \theta C +D):\ \theta \in I\}$  is equal to the model
$\{ Q( \theta pC +pD,\mu_p):\ \theta \in I\}$ parametrized by the means  and
 we have
 $k_{\mu}(\psi(\theta pC+pD))=p\log\det(\theta C+D).$
 \end{proof}

\begin{cor}
 Let $\sigma_1, \sigma_2\in \mathcal{S}^+_d$ and let $I$ be the open interval containing $\theta$ such that
$\sigma_\theta=\theta\sigma_1+(1-\theta)\sigma_2\in S^+_d.$
The Fisher information of the  model  $\{ \gamma(p; \sigma_\theta):\ \theta\in I\}$
is equal to
$
 J(\theta)= p\tr\left( \left((\sigma_1-\sigma_2)\,\sigma_t^{-1} \right)^2\right). 
$
\end{cor}
\begin{proof}
We write $\theta\sigma_1+(1-\theta)\sigma_2= \theta(\sigma_1-\sigma_2) + \sigma_2$
and we apply  formula (\ref{wish1}). 
\end{proof}
Using (\ref{wish2}) we obtain the following corollary, analogous to Corollaries
\ref{corcircul} and \ref{cortridiag}. 
\begin{cor}
1. Consider  the model  $\{ \gamma(p; \theta A +I_d):\ \theta \in I\}$  with $\theta A +I_d$ as in (\ref{firstcase}). Then its Fisher information
equals $J(\theta)=p\sum_{j=0}^{d-1} \left(\frac{2\cos(\frac{2\pi j}{d}) }{1+2\theta \cos(\frac{2\pi j}{d})}\right)^2.$\\
2. Consider  the model  $\{ \gamma(p; \theta C +I_d):\ \theta \in I\}$  with $\theta C +I_d$ as in (\ref{secondcase}). Then its Fisher information
equals $J(\theta)=p\sum_{j=1}^d\left(\frac{2\cos\left(\frac{j}{d+1}\pi \right)}{1+2\theta\cos\left(\frac{j}{d+1}\pi \right)}\right)^2.$
\end{cor}
\begin{rem}\label{GaussWishart}
Let $P(s,\mu)$ be the natural exponential family corresponding to 
the Gaussian general exponential family (\ref{eqt1}).  If $W$ has the law  $N(u,s^{-1})$ given by (\ref{eqt1}),
then $T(W)$ has the law $P(s,\mu)$. On the other hand it is well known that $-T(W)=\frac12 (W-u) (W-u)^T$ has the Wishart law $\gamma(\frac12; 2s^{-1})$. 
This explains why the formulas for the Fisher information are the same for  the Gaussian family and for the  Wishart family with $p=\frac12$. 
\end{rem}
$\;$\\
{\bf Exponential families of noncentral Wishart distributions}. 
Let us finish the section on the Wishart models by considering the non-central case. The main reference is
 \citep{letac_noncentral}.

Let $p\in\Lambda_d$, $a\in \overline{\mathcal{S}_d^+}$ and $\sigma \in \mathcal{S}_d^+$.
The noncentral  Wishart distribution
$\gamma(p, a; \sigma)$ is defined by its Laplace transform 
\begin{equation*}
\LL \gamma(p, a; \sigma) (t)= \int_{\overline{\mathcal{S}_d^+}} e^{-\tr(t x)}\gamma(p,a;\sigma)(dx)
=\det(I_d+\sigma t)^{-p}e^{-\tr\left(t(I_d+\sigma\,t )^{-1}\sigma a \sigma\right)},
\end{equation*}
for all $t\in  \mathcal{S}_d^+$. When $p\ge \frac{d-1}{2}$, then  non-central Wishart laws exist for all  $a\in \overline{\mathcal{S}_d^+}$;
    when $p\in  \{\frac{1}{2},\hdots,\frac{d-2}{2}\}$ then $a$ must be of rank at most $2p$ \citep{letac_existence}.
When $p=\frac{n}2$, $n\in \N$,  the  non-central Wishart distributions are constructed in the following way from
 $n$ independent $d$-dimensional Gaussian vectors
 $Y_1,\hdots, Y_n$. Let
$Y_j \sim N(m_j\,,\, \Sigma)$ and let $M$ be the $d\times n$ matrix $[m_1,\hdots,m_n]$.
Then, the $d\times d$ matrix 
$W = Y_1Y_1^T + \hdots + Y_nY_n^T$ has the noncentral Wishart distribution
$\gamma(p, a; \sigma)$ with $p=\frac{n}{2}$, $\sigma=2\Sigma$ and 
$\sigma a \sigma = MM^T$. Such Wishart distributions are studied in \citep{Muir}.

The  non-central Wishart distributions may be constructed as
   a natural exponential
family $\{W(p,a; s):\ s\in -\mathcal{S}_d^+\}$ generated by the positive measure 
$\mu=\mu_{a,p}(dx)=e^{\tr(a+ x)}\gamma(p,a ; I_d)(dx)$. 
Its moment generating function 
is given for $ s\in -\mathcal{S}_d^+$ by
\begin{equation*}
 \int_{\overline{\mathcal{S}_d^+}} e^{\tr(s x)}\mu_{a,p}(dx)
=\det(-s)^{-p}e^{\tr\left(a(-s) ^{-1}\right)}.
\end{equation*}
We have  $W(p,a; s)=\gamma(p, a; (-s)^{-1})$.
Like for central Wishart families, $S_{\mu} = -\mathcal{S}_d^+$.
The cumulant generating function is 
\begin{equation*} 
k_{\mu}(s) =-p \log \det(-s)+\tr(a(-s)^{-1}).
\end{equation*}

As before, we  denote $\sigma=(-s)^{-1}$.
We see that the mean equals
\begin{equation}\label{meanNonCentral}
 m(s)=k'_{\mu}(s) =p (-s)^{-1}+(-s)^{-1} a(-s)^{-1}= p\sigma+\sigma a \sigma  
\end{equation}
and the covariance
\begin{equation}\label{varianceNonCentral}
v(s)=k''_\mu(s)=p \sigma \otimes \sigma + (\sigma a \sigma)\otimes\sigma+ \sigma\otimes (\sigma a \sigma)
=-p \sigma \otimes \sigma + m\otimes\sigma+ \sigma\otimes m.
\end{equation}
When the matrix $a$ is non-singular, the inverse mean map $\psi(m)=s$ is such that
\begin{equation}\label{inverseBIS}
(-s)^{-1}=\sigma=-\frac{p}2 a^{-1} + a^{-1/2}\left( a^{1/2}ma^{1/2} + \frac{p^2}4I_d\right)^{1/2} a^{-1/2}. 
\end{equation}
For  other cases see   \citep[Prop.4.5]{letac_noncentral}. In order to write the variance function
$V(m)=v(\psi(m))$ we compose the last expression from  (\ref{varianceNonCentral}) and   the formula (\ref{inverseBIS}).

For a model $\{ W(p,a;\psi(\theta A+B)):\, \theta\in I\}$   parametrized by a segment of means, the   Fisher information $J(\theta)$ is obtained from the expression of $V(m)$
and Theorem \ref{thx}.\\[1mm]
{\bf Example.} Suppose that $a=I_d$, $A= \alpha I_d$ and $B=\beta I_d$, $\alpha,\beta>0$.
The Fisher information on  $\theta$ is 
$ J(\theta)=\alpha^2d\left( (p^2+2\theta\alpha +2\beta)(\theta\alpha+\beta+\frac{p^2}4)^{1/2} -2p(\theta\alpha+\beta)-\frac{p^3}2\right)^{-1}.$

\subsection{Estimation of the mean in exponential families parametrized by a segment of means}

Consider a sample $X_1,\ldots, X_n$ of a random variable $X$ from a natural  exponential family $Q(m,\mu)$ parametrized by the domain of means $M$,
where the parameter $m=\E X$ is unknown and $M$ is open. The following  qualities  of the sample mean $ \bar X_n$ as an estimator
of $m$ seem to be known; for the sake of completeness we provide a short proof of properties which are less evident.
\begin{pro}
 The sample mean $ \bar X_n$ is an  unbiased, consistent and efficient
 estimator of the parameter $m$.  It is also  a maximum likelihood  estimator of  $m$.
\end{pro}
\begin{proof} 
 By Theorem  \ref{th1}  we have $\cov X= V(m)= J(m)^{-1}$, so the Cram\'er-Rao bound is attained by $X$.  Consequently,    the sample mean $ \bar X_n$ is an   efficient estimator of $m$. 
It follows from  (\ref{param2}) that the sample mean $ \bar X_n$ is a  maximum likelihood  estimator of $m$.  One can also first show by (\ref{param1}) that 
 the maximum likelihood estimator of $s$ is 
 $
 \hat{s} = k'^{-1}_{\mu}(X) = \psi(X)$
 and next use the functional invariance of the maximum likelihood estimator
\citep[Theorem 7.2.10]{casella}.
\end{proof}

\begin{rem}
For general exponential families $Q(m,T,\nu)$ parametrized by an open domain of means $M$, all these properties remain valid  
for $\hat m= \overline{T(X)}_n$ as an estimator of $m=\E T(X)$.
\end{rem}

Consider an  exponential family   $Q( \theta A+B,\mu)$   parametrized by a segment of means $IA  +B\subset M$ with $A\not=0, B\in E$ and  $\theta\in I$, a segment in $\R$.
We  will now discuss estimators of the real parameter $\theta$
when we know that the mean $\E X= m\in I A  +B$.

 The segment $IA +B\subset M$ is of dimension one and has an empty interior. That's why the efficiency and maximum likelihood properties of  the estimator  $\hat m= \bar X_n$
 are not automatically inherited by natural estimators of the real parameter $\theta$. Determining  a   maximum likelihood estimator
 for $\theta$ seems impossible explicitly.  This is the "price to pay" for the parsimony of the segment  model parametrized by
$m\in  I A  +B$. On the other hand, the efficiency of estimators of $\theta$  may be studied thanks to Theorem \ref{thx} and its corollaries.\\

 Knowing  that
 \begin{equation}\label{solving}
   m=\theta  A+B
 \end{equation}
 for a value $\theta\in I$, we have many possibilities of  writing down a solution $\theta$ of   equation (\ref{solving}).
 If $A\not=0$ then the solution  $\theta$ is unique ($A\theta +B=A\theta' +B$ implies $\theta=\theta'$
when $A\not=0$.)
 For any $C$ such that $\langle A \,,\,C\rangle\not=0$ we  have
 \begin{eqnarray*}
\theta&=& \frac{\langle m-B\,,\, C\rangle}{\langle A \,,\,C\rangle }.
\end{eqnarray*}
We define an estimator $\hat \theta_C$ of the parameter $\theta$ by
\begin{eqnarray*}
\hat \theta_C&=& \frac{\langle \bar X_n-B\,,\, C\rangle}{\langle A \,,\,C\rangle }
\end{eqnarray*}

 All the  estimators $\hat \theta_C$ are unbiased and consistent.
The natural question is  whether  they are efficient.
The variance of $\hat \theta_C$ may be computed using the variance function $V(m)$ of the exponential family:
\begin{equation}\label{varianceTheta}
 {\rm Var}\, \hat \theta_C=\frac1{\langle A ,C\rangle^{2} } {\rm Var}\langle \bar X_n, C\rangle=\frac{\vc(C)^T V(\theta A+B)\vc (C)}{n\langle A ,C\rangle^{2} }.
\end{equation}
On the other hand, the Cram\'er-Rao bound is equal by  Theorem \ref{thx} to
\begin{equation}\label{CRaobound}
 \frac1{nJ(\theta)}= \frac1{n \vc(A)^T\, V(\theta A+B)^{-1}\,\vc(A)}.
\end{equation}

When the space $E$ is a squared matrix space $\R^{d\times d}$ and the matrix $A$ is invertible, we can take $C=A^{-1}$ and  consider the estimator
$$\hat \theta_{A^{-1}}=\frac{\langle \bar X_n-B, A^{-1}\rangle }{d}.$$
The   following theorem shows that for Gaussian and central Wishart exponential families and for 
linearly dependent $A$ and $B$  the estimator $\hat\theta_{A^{-1}}$ is efficient as an  estimator of the mean $m$
(with  $X_i$ replaced by $T(X_i)=-\frac12(X_i-u)(X_i-u)^T$ in  the Gaussian case). 
In conclusion, we obtain efficient estimators for Gaussian models parametrized by a
  covariance segment parameter  and for Wishart models parametrized by a
   scale segment parameter.
\begin{thm}
1.  Let  $I\subset \R^+$ be a non-empty segment.   Let $c\ge 0$,  $A\in \mathcal{S}_d^+$ and $B=cA$.  \\
(1a) Consider an  $n$-sample $(X_1,\ldots,X_n)$ from a  Gaussian family $Q(m,T,\nu)$  defined by (\ref{GaussMean}),
 where  $m=\theta A+B$, $\theta\in I$.
Then $$\hat \theta_{A^{-1}}=\frac{\langle \overline{T(X)}_n-B, A^{-1}\rangle }{d}$$
is an unbiased efficient estimator  of the parameter $\theta$.\\
(1b) Consider an $n$-sample $(X_1,\ldots,X_n)$ from a  Wishart model $Q(m, \mu_p)$   defined by (\ref{WishartMean}),
 where  $m=\theta A+B$, $\theta\in I$.
  Then 
$$\hat \theta_{A^{-1}}=\frac{\langle \bar X_n-B, A^{-1}\rangle }{d}$$
is an unbiased efficient estimator  of the parameter $\theta$.\\
2. Let $c\ge 0$,  $C\in \mathcal{S}_d^+$ and $D=cC$.\\
(2a)  Consider an  $n$-sample $(X_1,\ldots,X_n)$ from  a Gaussian model $\{N(u, \theta  C  + D),  \theta \in I\}$
parametrized by a segment of covariances.
An unbiased efficient estimator of $\theta$ is given by 
$$\hat \theta=\frac{1}{d} \langle \frac{1}{n}\sum_{i=1}^n (X_i-u) (X_i-u)^T-D, C^{-1}\rangle. $$
$\;$\\
(2b)  Consider an  $n$-sample $(X_1,\ldots,X_n)$ from  a Wishart  model $\{\gamma(p,\theta  C  + D) ,  \theta \in I\}$
parametrized by a segment of  scale parameters.
Un unbiased efficient estimator of $\theta$ is given by 
$$\hat \theta=\frac{\langle \frac1p \bar X_n-D, C^{-1}\rangle }{d}.$$
\end{thm}
\begin{proof}
For the first part of the Theorem, we give the proof in the Wishart case. The proof in the Gaussian case is identical, with $p=\frac12$, cf. Remark \ref{GaussWishart}.
 By  formulas (\ref{varianceTheta}) and  (\ref{VarWish}) 
 $${\rm Var}\, \hat \theta_{A^{-1}}=\frac{1}{p d^2 n} \tr((A\theta+B)A^{-1}(A\theta+B)A^{-1})=\frac{(\theta+c)^2}{p d n}
 $$
 On the other  hand, by (\ref{CRaobound}) and  (\ref{VarWish}) 
 $$
 \frac1{nJ(\theta)}=\frac{1}{np   \tr(A(A\theta+B)^{-1}A(A\theta+B)^{-1})}=\frac{1}{np (\theta+c)^{-2 }d}.
 $$
 Thus ${\rm Var}\, \hat \theta=\frac1{nJ(\theta)}$ and the estimator $\hat\theta_{A^{-1}}$ is efficient.
 
 The second part of the Theorem follows by necessary reparametrizations. For (2a), using (\ref{GaussMean}), we write $\theta  C  + D= -2m$ with
 $m=\theta A + B$, where $A=-\frac{C}2$ and  $B=-\frac{D}2$. The part (2b) follows similarly from (\ref{WishartMean}).
\end{proof}
\begin{rem}
It is an open question whether $\hat\theta_{A^{-1}}$ may be efficient for independent $A$ and $B$.
 Let $n=1$. The equality ${\rm Var}\, \hat \theta = \frac{1}{J(\theta)}$ holds if and only if,
writing
$D_\theta=(A\theta+B)A^{-1}(A\theta+B)A^{-1},$
the  equality $\frac{1}{d^2} \tr(D_\theta)=\frac{1}{ \tr(D_\theta^{-1})}$ holds for all $\theta\in I$.
\end{rem}

 {\bf Acknowledgements.} The authors would like to thank   G\'erard Letac for his insightful suggestions.
 The second author   acknowledges  financial support(co-tutelle fellowship) from the French Ministry of Foreign Affairs
 and the Embassy of France in South Africa. The authors   are greatly indebted to l'Agence Nationale de la Recherche  for  the research grant ANR-09-Blan-0084-01.

 \bibliographystyle{mystyle}
\bibliography{references}

\end{document}